\PassOptionsToPackage{breaklinks}{hyperref}
\documentclass[12pt, draftcls, oneside, onecolumn, technote, letter]{IEEEtran}

\usepackage{graphicx}
\usepackage{amsmath}
\usepackage{hyperref}
\usepackage{amssymb}                            % document contains figures,
\usepackage[dvips]{epsfig}
\usepackage{subfig}
\newtheorem{definition}{Definition}
\newtheorem{assumption}{Assumption}
\newtheorem{theorem}{Theorem}
\newtheorem{remark}{Remark}

\newtheorem{lem}{Lemma}

\newcommand{\Min}{\displaystyle\min}

\newcommand{\Sum}{\displaystyle\sum}

\title{\LARGE \bf
Inertial Measurements Based Velocity-free Attitude Stabilization }
\author{
A. Tayebi, A. Roberts, A. Benallegue% <-this % stops a space
\thanks{
Partial and preliminary results related to this work have been presented in \cite{Tayebi_ACC_2011}.
This work is supported by the Natural Sciences and Engineering Research Council of Canada (NSERC).}
\thanks{A. Tayebi and A. Roberts are with the Department of Electrical and Computer Engineering, University of Western Ontario, London, Ontario, Canada.
A. Tayebi is also with the Department of Electrical Engineering, Lakehead University, Thunder Bay, Ontario, Canada.
A. Benallegue is with LISV, Versailles, France.
        {\tt\small atayebi@lakeheadu.ca, arober88@uwo.ca, benalleg@lisv.uvsq.fr} }
}

\begin{document}
\maketitle \thispagestyle{empty} \pagestyle{empty}

\begin{abstract}
The existing attitude controllers (without angular velocity measurements) involve explicitly the orientation (\textit{e.g.,} the unit-quaternion) in the feedback. Unfortunately, there does not exist any sensor that directly measures the orientation of a rigid body, and hence, the attitude must be reconstructed using a set of inertial vector measurements as well as the angular velocity (which is assumed to be unavailable in velocity-free control schemes). To overcome this \textit{circular reasoning}-like problem, we propose a velocity-free attitude stabilization control scheme relying solely on inertial vector measurements. The originality of this control strategy stems from the fact that the reconstruction of the attitude as well as the angular velocity measurements are not required at all. Moreover, as a byproduct of our design approach, the proposed controller does not lead to the unwinding phenomenon encountered in unit-quaternion based attitude controllers.
\end{abstract}

\section{Introduction}
The attitude control problem of rigid bodies has been widely studied over the last decades. The interest devoted to this problem is motivated by its technical challenges as well as its practical implications in many areas such aerospace and marine engineering. The main technical difficulty encountered in this type of mechanical systems may be attributed to the fact that the orientation (angular position) of the rigid body is not a straightforward integration of the angular velocity. Several interesting solutions to the attitude control problem have been proposed in the literature (see, for instance, \cite{Chaturvedi}, \cite{Lizarraide}, \cite{TayebiTAC08}, \cite{Wen}). \\
As it is customary in the position control of mechanical systems, the majority of the control schemes developed for rigid bodies are (roughly speaking) of Proportional-Derivative (PD) type, where the proportional action is in terms of the orientation and the derivative action (generating the necessary damping) is in terms of the angular velocity. The requirement of the angular velocity can be removed through an appropriate design, usually based on the passivity properties of the system,
as done in \cite{Costic}, \cite{Egeland}, \cite{Lizarraide}, \cite{Salcudean}, \cite{TayebiTAC08} and \cite{Tsiotras}, for instance.
Since there is no sensor that directly measures the orientation, the explicit use of the attitude (\textit{e.g.,} the unit-quaternion) in the control law calls for
the development of suitable attitude estimation algorithms that reconstruct the attitude from the
measurements provided by some appropriate sensors (such as inertial measurements units (IMUs)).
Therefore, the integration of the attitude estimation algorithm with the attitude controller has to be taken with extra care as the separation principle does not generally hold for non-linear systems.

Initially, the attitude determination  from vector observations, has been tackled as a static optimization problem for which several solutions, based on Wahba's problem, have been proposed \cite{Shuster2}. These algorithms have been refined, later on, incorporating filtering techniques of Kalman-type to handle the measurement noise (see, for instance, \cite{Choukroun, Shuster3} and the survey paper \cite{Crassidis}).
On the other hand, probably the most simple and yet practical dynamic attitude estimation approach is the linear complementary filtering (see, for instance, \cite{Tayebi}),
where the angular velocity and the inertial vector measurements are fused (via some appropriate linear filters) to recover the orientation of the rigid body for small angular movements.
This approach has been extended to nonlinear complementary filtering for the attitude estimation
from vector measurements in \cite{Hua,Mahony,Martin1,Martin2,Roberts}.

In fact, the existing dynamic estimation algorithms (involving signals filtering) make use of the angular velocity information to reconstruct the orientation of the rigid body. Therefore, a natural question that may arise is whether it makes sense to use velocity-free attitude controllers such as those proposed, for instance, in \cite{Lizarraide}, \cite{TayebiTAC08} and \cite{Tsiotras}, since the angular velocity will be used (anyways) to recover the attitude via dynamic estimation algorithms. In this context, the main contribution of the present paper, is the development of a new attitude stabilization control scheme that uses explicitly inertial vector measurements without requiring (either directly or indirectly) the velocity measurement. This controller is  a \textit{true velocity-free} scheme since neither the angular velocity nor the orientation are used in the control law. Moreover, as it will be shown later, the unwinding phenomenon\footnote{Using the unit-quaternion representation, the same attitude can be represented by two different unit-quaternion, namely $Q$ and $-Q$. This fact, often leads to the so-called unwinding phenomenon, which is undesirable since some closed-loop trajectories, for some initial conditions close to the desired attitude equilibrium, can undergo an unnecessary homoclinic-like orbit \cite{Bhat,Chaturvedi}.} (inherent to the quaternion representation) is avoided in our approach.

\section{Rigid-body attitude representation and equations of motion}

The attitude (orientation) of a rigid body can be represented in several ways, among which, one can recall the rotation matrix $R$ evolving in the special orthogonal group of degree three,
\[
SO(3)=\{R \in \mathbb{R}^{3 \times 3}~|~R^TR=RR^T=I_{3}, det(R)=1\},
\]
where $I_3$ is a 3-by-3 identity matrix.
Another globally non-singular representation of the attitude consists of using four-dimensional vectors $Q$, called unit-quaternion, evolving in the three-sphere $\mathbb{S}^3$, embedded in $\mathbb{R}^{4}$, $\mathbb{S}^3=\{Q \in \mathbb{R}^{4}~|~Q^TQ=1\}$.\\
A unit-quaternion $Q=(q_0,q^T)^T$ is composed of a scalar component $q_0 \in \mathbb{R}$ and a vector component $q \in \mathbb{R}^{3}$, such that $q_0^2+q^Tq=1$. A rotation matrix $R$ describing a rotation by an angle $\theta$ about the unit-vector $\hat{k}\in \mathbb{R}^3$, can be represented by the unit-quaternion $Q$ or $-Q$ such that $q_0=\cos(\theta/2)$ and $q=\sin(\theta/2)\hat{k}$. Note that the mapping from $SO(3)$ to $\mathbb{S}^3$ is not a one-to-one mapping as there are two unit quaternion that represent the rotation matrix $R$.
The rotation matrix $R$ (that describes the orientation of the body-attached frame with respect to the inertial frame) is related to the unit-quaternion $Q=(q_0,q^T)^T$ through the Rodriguez formula $R=\mathcal{R}(Q)$ (see, for instance, \cite{Hughes, Shuster}). The mapping $\mathcal{R}: \mathbb{S}^3\rightarrow SO(3)$ is given by
\begin{equation}\label{R}
\begin{array}{rcl}
\mathcal{R}(Q)&=&I_3+2q_0 S(q)+2S(q)^2=(q_0^{2}-q^{T}q)I_3+2qq^{T}+2q_0 S(q),
\end{array}
\end{equation}
where $S:\mathbb{R}^3 \rightarrow \mathfrak{so}(3)$, such that
\[
S(x)=\left[\begin{array}{ccc}
0&-x_3& x_2\\
x_3& 0&-x_1\\
-x_2& x_1& 0\\
\end{array} \right].
\]
with $x=[x_1,x_2,x_3]^T \in \mathbb{R}^3$, and $\mathfrak{so}(3)=\{S \in\mathbb{R}^{3\times3}~|~S^T=-S\}$ is the set of 3-by-3 skew symmetric matrices. Given a rotation matrix $R$ and two vectors $x,y\in\mathbb{R}^3$, we have the following useful properties: $S(x)y=-S(y)x=x\times y$, $S(x)x=0$, $S(x)S(y)=yx^T-(x^Ty)I_3$ and $S(Rx)=RS(x)R^T$, where $\times$ denotes the vector cross product.\\
The three-sphere (where the unit-quaternion evolve) has a natural Lie group structure given by the quaternion multiplication (which is distributive and associative but not commutative) denoted by ``$\odot$''. The multiplication of two quaternion $P=(p_0,p^T)^T$ and $Q=(q_0,q^T)^T$ is defined as
\begin{equation}
P \odot Q=\left[\begin{array}{c}
p_0q_0-p^Tq\\
p_0q+q_0p+p\times q
\end{array}\right]
\label{quatmul}
\end{equation}
and has the quaternion $(1,\bold{0}_3^T)^T$ as the identity element. Given a unit-quaternion $Q=(q_0,q^T)^T$, we have $Q\odot Q^{-1}=Q^{-1}\odot Q=(1,\bold{0}_3^T)^T$, where $Q^{-1}=(q_0,-q^T)^T$ and $\bold{0}_3$ is 3-dimensional zero vector.

Throughout this paper, $\bar{x}:=(0,x^T)^T$ will denote the quaternion associated to the three-dimensional vector $x$. A vector $x_{\mathcal{I}}$ expressed in the inertial frame $\mathcal{I}$ can be expressed in the body frame $\mathcal{B}$ by $x_{\mathcal{B}}=R^T x_{\mathcal{I}}$ or equivalently in terms of unit-quaternion as $\bar{x}_{\mathcal{B}}=  Q^{-1}\odot \bar{x}_{\mathcal{I}} \odot Q $, where $\bar{x}_{\mathcal{I}}=(0, x_{\mathcal{I}}^T)^T$, $\bar{x}_{\mathcal{B}}=(0, x_{\mathcal{B}}^T)^T$, and $Q$ is the unit-quaternion associated to $R$ as per (\ref{R}). \\
The rigid body attitude kinematics, in terms of the rotation matrix, are given by $\dot{R}=RS(\omega)$, and, in terms of the unit-quaternion, by
\begin{equation}\label{quat_dyn}
\dot{Q}:=\left[ \begin{array}{c}
\dot{q}_0\\
\dot{q}
\end{array}\right] =\frac{1}{2} Q \odot\bar{\omega}=\left[
  \begin{array}{c}
  -\frac{1}{2}q^T\omega\\
  \frac{1}{2}(q_0I_3+S(q))\omega
  \end{array} \right]
 \end{equation}
with $\omega$ being the angular velocity of the rigid body expressed in the body-attached frame
$\mathcal{B}$.\\
The rigid body rotational dynamics are governed by
\begin{equation}\label{rot_dyn}
J\dot{\omega}=\tau-S(\omega)J\omega,
 \end{equation}
were $J\in\mathbb{R}^{3\times3}$ is a symmetric positive definite constant inertia matrix of rigid body with respect to the body attached frame $\mathcal{B}$, and $\tau$ is the external torque applied to the system expressed in $\mathcal{B}$.

\section{Problem Statement}\label{section_attitude_estimation}
  We assume that the rigid body is equipped with at least two inertial sensors that provide measurements (in the body-attached frame) of constant and known inertial vectors $r_i\in\mathbb{R}^3$, $i=1,\hdots,n\geq2$, satisfying the following assumption:
  \begin{assumption}\label{assum}
 At least two vectors, among the $n$ inertial vectors, are not collinear.
  \end{assumption}
\vskip 0.08in
The vector measurements in the body-attached frame are denoted by $b_i\in\mathbb{R}^3$, $i=1,\hdots,n$. The vectors $r_i$ and $b_i$ are related by $b_i=R^Tr_i$. We assume that the angular velocity $\omega$ of rigid body as well as the attitude ($Q$ or $R$) are unknown (unavailable for feedback). The only variables available for feedback are the inertial vector measurements. Our objective is to design a control input $\tau$ (using only the vector measurements) guaranteeing almost global asymptotic stability of the equilibrium points characterized by $(q_0=\pm1,q=0,\omega=0)$ or, equivalently, $(R=I_3,\omega=0)$. Since the motion on $SO(3)$ (which is not a contractible space), is hampered by some well known topological obstructions that preclude global asymptotic stability results via time-invariant continuous state feedback (see, for instance, \cite{Bhat, Koditschek}), the notion of almost global asymptotic stability is commonly used in this context and is defined as follows:
\begin{definition}
An equilibrium point $x^*$ of a dynamical system is said to be almost globally asymptotically stable if it is stable and all the trajectories starting in some open dense subset of the state space converge to $x^*$.
\end{definition}

\section{Main results}
\subsection{Preliminary results}
  Let us define the following dynamic auxiliary system:
   \begin{equation}\label{obser1}
  \dot{\hat{Q}}= \frac{1}{2}\hat{Q} \odot \bar{\beta},
  \end{equation}
  with an arbitrary initial condition $\hat{Q}(0)\in \mathbb{S}^3$, where $\bar{\beta}:=(0,\beta^T)^T$ with $\beta$ being a design variable to be defined later in Theorem 1. \\
  Let us define the vectors $\hat{b}_i$, $i=1,\hdots,n$, as $\hat{b}_i=\mathcal{R}(\hat{Q})^T r_i$ or, equivalently, $\bar{\hat{b}}_{i}=\hat{Q}^{-1} \odot \bar{r}_{i} \odot \hat{Q}$, corresponding to the vector $r_i$ in the frame attached to the auxiliary system (\ref{obser1}).\\
  Let $\tilde{R}:=\mathcal{R}(\tilde{Q})=\mathcal{R}(Q)\mathcal{R}(\hat{Q})^T $ denote the discrepancy between the actual rigid-body orientation and the orientation of the auxiliary system (\ref{obser1}), which corresponds to the unit quaternion error $\tilde{Q}:=(\tilde{q}_0, \tilde{q}^T)^T=Q\odot\hat{Q}^{-1}$ whose dynamics is governed by
  \begin{equation}\label{err_quat}
  \begin{array}{rcl}
  \dot{\tilde{Q}}&=&\frac{1}{2}\tilde{Q} \odot \bar{\tilde{\omega}}=\left[
  \begin{array}{c}
  -\frac{1}{2}\tilde{q}^T \tilde{\omega}\\
  \frac{1}{2}(\tilde{q}_0I_3+S(\tilde{q}))\tilde{\omega}
  \end{array} \right]
 \end{array}
  \end{equation}
 where $\tilde{\omega}=\mathcal{R}(\hat{Q})(\omega-\beta)$. \\
In the sequel, for the sake of simplicity, the arguments associated to the rotation matrices will be omitted. It is understood that $R$, $\hat{R}$ and $\tilde{R}$ correspond, respectively, to the rotation matrices $\mathcal{R}(Q)$, $\mathcal{R}(\hat{Q})$ and $\mathcal{R}(\tilde{Q})$, associated, respectively, to the unit quaternion $Q$, $\hat{Q}$ and $\tilde{Q}$.\\
Before stating our main results, let us define the following variables:  $z_\gamma:=\sum_{i=1}^n \gamma_iS(\hat{b}_i)b_i$, $z_\rho:=\sum_{i=1}^n \rho_i S(r_i)b_i$, $W_{\gamma}:=-\sum_{i=1}^{n} \gamma_i S(r_i)^2$ and $W_{\rho}:=-\sum_{i=1}^{n} \rho_i S(r_i)^2$, where $\gamma_i>0$ and $\rho_i>0$ for $i=1,\hdots,n$, and let us state the following useful Lemmas:
\begin{lem}\label{lem1}
The following equalities hold:
\begin{equation}\label{zgamma}
z_\gamma=-2\hat{R}^T(\tilde{q}_0I-S(\tilde{q}))W_{\gamma}\tilde{q},
\end {equation}
\begin{equation}\label{zrho}
z_\rho=-2(q_0I-S(q))W_{\rho} q,
\end{equation}
\end{lem}
\begin{proof}
Using the facts that $b_i=R^Tr_i$, $\hat{b}_i=\hat{R}^Tr_i$ and $S(\hat{R}^Tr_i)=\hat{R}^TS(r_i)\hat{R}$, we have
\begin{equation}\label{z_gamma_proof}
z_\gamma:=\sum_{i=1}^n \gamma_iS(\hat{b}_i)b_i=\hat{R}^T\sum_{i=1}^n \gamma_iS(r_i)\tilde{R}^Tr_i.
\end{equation}
Rewriting $\tilde{R}$ in terms of the unit-quaternion $\tilde{Q}$, \textit{i.e.,} $\tilde{R}=\mathcal{R}(\tilde{Q})$, and using the facts that $S(r_i)r_i=0$ and $S(\tilde{q})r_i=-S(r_i)\tilde{q}$, one gets
\begin{equation}\label{zgproof}
z_\gamma=2\hat{R}^T\sum_{i=1}^n \gamma_iS(r_i)(\tilde{q}\tilde{q}^{T}-\tilde{q}_0 S(\tilde{q}))r_i=-2\hat{R}^T(S(\tilde{q})M_\gamma+\tilde{q}_0W_{\gamma})\tilde{q},
\end{equation}
where $M_\gamma=\sum_{i=1}^n \gamma_i r_ir_i^T$. One can also show that $M_\gamma=\mu I_3-W_{\gamma}$, $\mu=\sum_{i=1}^n \gamma_i r_i^Tr_i$. Substituting the expression of $M_\gamma$ in (\ref{zgproof}), yields (\ref{zgamma}).
Similar steps can be used to obtain $z_\rho$.
\end{proof}

\begin{lem}\label{lem2}
Under Assumption \ref{assum}, the matrices $W_{\gamma}$ and $W_{\rho}$ are positive definite.
\end{lem}
\begin{proof}
For any $y\in\mathbb{R}^3$ we have
\[
-y^TS(r_i)^2y=y^T(r_i^Tr_i)y-y^Tr_ir_i^Ty=\|r_i\|^2\|y\|^2-(y^Tr_i)^2 \geq 0,
\]
which shows that $-S(r_i)^2 \geq 0$. Now, without loss of generality, assuming that $r_1$ and $r_2$ are the two non-collinear vectors, one has
\begin{equation}\label{positive}
y^TW_{\gamma}y=-y^T\left(\sum_{i=1}^{n}\gamma_iS(r_i)^2\right)y=-\gamma_1y^TS(r_1)^2y-\gamma_2y^TS(r_2)^2y+y^T\Xi y,
\end{equation}
where $\Xi=-\sum_{i=3}^n \gamma_i S(r_i)^2 \geq 0$. It is clear that $y^TS(r_1)^2y=0$ is equivalent to $S(r_1)y=r_1 \times y=0$, which is satisfied, for some $y\not=0$, if and only if $y$ and $r_1$ are collinear. Similarly, for $y\not=0$, $y^TS(r_2)^2y=0$ if and only if $y$ and $r_2$ are collinear. Since $r_1$ and $r_2$ are non-collinear, they cannot be both collinear to the same $y$, and hence,  $-\gamma_1y^TS(r_1)^2y-\gamma_2y^TS(r_2)^2y>0$ for all $y\not=0$. Consequently, from (\ref{positive}), one can conclude that $y^TW_{\gamma}y>0$ for all $y\not=0$, which ends the proof.
\end{proof}

\begin{lem}\label{lem3}
Under Assumption \ref{assum}, the following statements hold:
\begin{itemize}
\item[$i)$] $z_\gamma=0$ is equivalent to $(\tilde{q}_0=0,\tilde{q}= v_{\gamma})$  or $(\tilde{q}_0=\pm1,\tilde{q}=0)$,
            were $v_{\gamma}$ are the unit eigenvectors of $W_{\gamma}$.
\item[$ii)$]    $z_\rho=0$ is equivalent to $(q_0=0,q=v_{\rho})$ or $(q_0=\pm1, q=0)$, were $v_{\rho}$ are the unit eigenvectors of $W_{\rho}$.
 \end{itemize}
 \end{lem}

\begin{proof}
Using Lemma 1, $z_\gamma=0$ is equivalent to
\begin{equation}\label{property}
(\tilde{q}_0I-S(\tilde{q}))W_{\gamma}\tilde{q}=0,
\end{equation}
from which one can see that $(\tilde{q}_0=\pm1,\tilde{q}=0)$ is a trivial solution. The other solutions can be obtained by assuming  $\tilde{q} \not=0$ and multiplying (\ref{property}) by $\tilde{q}^T$ to obtain $\tilde{q}_0\tilde{q}^T W_{\gamma} \tilde{q}=0$,
which shows that $\tilde{q}_0=0$ is a solution since $W_{\gamma}$ is positive definite (as per Lemma 2). Now, we need to find the values of $\tilde{q}$ for which $\tilde{q}_0=0$. Using (\ref{property}) with $\tilde{q}_0=0$, we obtain $S(\tilde{q})W_{\gamma}\tilde{q}=0$. Since $W_{\gamma}$ is non-singular, the last equality is satisfied if and only if $W_{\gamma}\tilde{q}=\lambda \tilde{q}$ for some constant scalar $\lambda$, which shows that $\tilde{q}=v_{\gamma}$, where $v_{\gamma}$ are the unit eigenvectors of $W_{\gamma}$. Similar steps can be used to prove $(ii)$ and hence omitted here. This completes the proof.
\end{proof}
%\begin{remark}
%Note that $v_{\gamma}$ (respectively $v_{\rho}$) represents the three unit eigenvectors of $W_{\gamma}$ (respectively $W_{\rho}$). This somewhat abusive notation is used for the sake of presentation simplicity since it does not lead to any confusion. The reader needs to keep in mind that whenever $v_\gamma$ (or $v_{\rho}$) is used throughout this paper, it represents indifferently the three unit eigenvectors.
%\end{remark}

\begin{remark}
It worth pointing out that the unit-eigenvectors of $W_{\gamma}$ are also the unit-eigenvectors of  $M_\gamma=\sum_{i=1}^n \gamma_i r_ir_i^T$. This can be easily deduced from the fact that $M_\gamma=\mu I_3-W_{\gamma}$, with $\mu=\sum_{i=1}^n \gamma_i r_i^Tr_i$. Likewise, the unit-eigenvectors of $W_{\rho}$ are also the unit-eigenvectors of  $M_\rho=\sum_{i=1}^n \rho_i r_ir_i^T$. Note also that the eigenvalues of $W_{\gamma}$ (resp. $W_{\rho}$) can be arbitrarily increased by increasing the gains $\gamma_i$ (resp. $\rho_i$).
\end{remark}

\subsection{Vector measurements based attitude stabilization}
We propose the following control law:
 \begin{equation}
\tau=z_\gamma+z_\rho,
\label{feed}
\end{equation}
where $z_\gamma$ and $z_\rho$ are defined in subsection IV-A, and the following input for the auxiliary system (\ref{obser1}):
\begin{equation}
\beta=-z_{\gamma}.
\label{beta}
\end{equation}
Under the proposed control law, the closed loop dynamics are given by
\begin{equation}\label{closed_loop_dynamics}
\begin{array}{rcl}
\dot{\tilde{q}}&=&\frac{1}{2}(\tilde{q}_0I_3+S(\tilde{q}))\hat{R}(\omega+z_\gamma)\\
\dot{q}&=&\frac{1}{2}(q_0I_3+S(q))\omega\\
J\dot{\omega}&=&-S(\omega)J\omega+z_\gamma+z_\rho,
\end{array}
\end{equation}
where $\hat{R}:=\mathcal{R}(\hat{Q})=\mathcal{R}(\tilde{Q})^T\mathcal{R}(Q)$, and
$z_\gamma=-2\mathcal{R}(Q)^T\mathcal{R}(\tilde{Q})(\tilde{q}_0I-S(\tilde{q}))W_{\gamma}\tilde{q}$ and $z_\rho=-2(q_0I-S(q))W_{\rho} q$. Note that the scalar parts of $Q$ and $\tilde{Q}$ that appear in (\ref{closed_loop_dynamics}) are related to the vector parts as follows: $q_0^2=1-q^Tq$ and  $\tilde{q}_0^2=1-\tilde{q}^T\tilde{q}$. It is clear that the closed loop dynamics (\ref{closed_loop_dynamics}) are autonomous.\\
Let $\chi:=(\tilde{q},q,\omega)^T$ be the state vector belonging to the state space $\Upsilon:=D\times D \times \mathbb{R}^3$, with $D:=\{x\in \mathbb{R}^3~|~\|x\|\leq 1\}$. Now, one can state our first theorem.

\begin{theorem}
Consider system (\ref{quat_dyn})-(\ref{rot_dyn}) under the control law (\ref{feed}). Let (\ref{beta}) be the input of (\ref{obser1}).
 Assume that $n$ vector measurements $b_i$, corresponding to the inertial vectors $r_i$, $i=1,\hdots,n\geq 2$ are available, and Assumption \ref{assum} holds. Then,
\begin{itemize}
\item[$1)$] The equilibria of the closed-loop system (\ref{closed_loop_dynamics}) are given by\footnote{Note that $\Omega_2$ is not a single equilibrium point as we have multiple values for $v_\gamma$ (unit eigenvectors of $W_\gamma$). The same remark goes for $\Omega_3$ and $\Omega_4$.}
  $\Omega_1=(0,0,0)$, $\Omega_2=(v_{\gamma},0,0)$, $\Omega_3=(v_{\gamma},v_{\rho},0)$ and $\Omega_4=(0,v_{\rho},0)$, where $v_\gamma$ and $v_\rho$ are, respectively, the unit eigenvectors of $W_\gamma$ and $W_\rho$.
\item[$2)$] The equilibrium point $\Omega_1$ is asymptotically stable with the domain of attraction containing the following domain\footnote{Note that the domain of attraction $\Phi_1$ includes the domain $\Phi_1^l:=\{\chi\in \Upsilon~|~\|\tilde{q}\|^2+\|q\|^2+\frac{\lambda_{max}(J)}{2\lambda_M}\|\omega\|^2<\frac{\lambda_m}{\lambda_M}$\}, with $\lambda_m=2\min\{\lambda_{min}(W_\gamma),\lambda_{min}(W_\rho)\}$, $\lambda_M=2\max\{\lambda_{max}(W_\gamma),\lambda_{max}(W_\rho)\}$, where $\lambda_{max}(\star)$ denotes the maximum eigenvalue of $(\star)$. If the control gains are sufficiently large such that $\lambda_M \gg \lambda_{max}(J)$, the domain $\Phi_1^l$ approaches $\{\chi\in \Upsilon~|~\|\tilde{q}\|^2+\|q\|^2<\frac{\lambda_m}{\lambda_M}$\}.}:
\[
\Phi_1=\left\{\chi\in \Upsilon~|~\chi^TP\chi \leq c \right\},
\]
where $P=\mbox{diag}(2W_\gamma,2W_\rho,\frac{1}{2}J)$, $c<2\min\{\lambda_{min}(W_\gamma),\lambda_{min}(W_\rho)\}$ and $\lambda_{min}(\star)$ is the minimum eigenvalue of $(\star)$.
\item[$3)$] There exist $k_{\rho}>0$ and $k_{\gamma}>0$ such that if $\lambda_{min}(W_{\rho}) >k_{\rho}$ and $\lambda_{min}(W_{\gamma}) >k_{\gamma}$, then $\Omega_2$, $\Omega_3$ and  $\Omega_4$ are unstable, and $\Omega_1$ is almost globally asymptotically stable.
\end{itemize}
\end{theorem}

\begin{proof}
Consider the positive definite, radially unbounded, function $V:\Upsilon\rightarrow  \mathbb{R}_{\geq 0}$  such that
\begin{equation}
V(\tilde{q},q,\omega)=2\tilde{q}^TW_{\gamma}\tilde{q}+2q^TW_{\rho}q+\frac{1}{2}\omega^T J \omega,
\label{Lyap}
\end{equation}
where $W_{\gamma}^T=W_{\gamma}>0$, $W_{\rho}^T=W_{\rho}>0$ are defined in subsection IV-A.
Using (\ref{R}) and the properties of the skew symmetric matrix and the unit-quaternion, one can show that $\Sum_{i=1}^{n} \gamma_i \tilde{b}_i^T  \tilde{b}_i=4\tilde{q}^TW_{\gamma}\tilde{q}$ and $\Sum_{i=1}^{n} \rho_i e_i^T e_i=4q^TW_{\rho}q $, where $\tilde{b}_i:=\hat{b}_i-b_i$ and $e_i:=r_i-b_i$.
One can also show that the dynamics of $\tilde{b}_i$ and $e_i$ are governed by
\begin{equation}
\begin{array}{rcl}
\dot{\tilde{b}}_i&=& S(\hat{b}_i)(\beta-\omega)+S(\tilde{b}_i)\omega\\
\dot{e}_i&=&-S(b_i)\omega,
\label{error}
\end{array}
\end{equation}
where we used the fact that $\dot{R}=RS(\omega)$ and $\dot{\hat{R}}=\hat{R}S(\beta)$.\\
Consequently, one can rewrite (\ref{Lyap}) as follows:
\begin{equation}
V=\frac{1}{2}\Sum_{i=1}^{n} \gamma_i \tilde{b}_i^T  \tilde{b}_i+\frac{1}{2}\Sum_{i=1}^{n} \rho_i e_i^T  e_i+\frac{1}{2}\omega^T J \omega,
\end{equation}
whose time-derivative, along the trajectories of (\ref{rot_dyn}) and (\ref{error}), is given by
\begin{equation}
\begin{array}{rcl}
\dot{V}&=&\Sum_{i=1}^{n} \gamma_i \tilde{b}_i^T S(\hat{b}_i )(\beta-\omega)-\Sum_{i=1}^{n} \rho_i e_i^T S(b_i)\omega+\omega^T\tau\\
~&=&\Sum_{i=1}^{n} \gamma_i \tilde{b}_i^T S(\hat{b}_i )(\beta-\omega)-\Sum_{i=1}^{n} \rho_i r_i^T S(b_i)\omega+\omega^T\tau=(\beta-\omega)^Tz_{\gamma}-\omega^Tz_{\rho}+\omega^T\tau,
\label{dotLyapunov}
\end{array}
\end{equation}
which, in view of (\ref{feed}) and (\ref{beta}), yields
\begin{equation}\label{dLyap_control}
  \dot{V}= -z_\gamma^T z_\gamma=-4\tilde{q}^TW_{\gamma}^T(I_3-\tilde{q}\tilde{q}^T)W_{\gamma}\tilde{q},
  \end{equation}
were we used Lemma 1 and the properties of the unit-quaternion and the skew-symmetric matrix to obtain the second equality. Since the closed-loop dynamics $(\dot{\tilde{q}},\dot{q},\dot{\omega})$ are autonomous, we proceed with LaSalle's invariance theorem. From (\ref{dLyap_control}), setting $\dot{V}\equiv 0$ leads to $z_\gamma \equiv0$, which in view of Lemma 3, implies that $\tilde{Q}\equiv(\pm1,\bold{0}_3^T)^T$ or $\tilde{Q}\equiv(0,v_{\gamma}^T)^T$. Since $\tilde{Q}$ is constant, one can conclude from (\ref{err_quat}) that $\omega-\beta\equiv 0$. Since $z_{\gamma}\equiv0$, it follows from (\ref{beta}) that $\beta\equiv0$, and consequently $\omega\equiv0$ since $\omega\equiv\beta$. Since $\omega\equiv0$, from (\ref{rot_dyn}), it follows that $\tau\equiv0$. Using this last fact, together with the fact that $z_\gamma\equiv0$, one can conclude from (\ref{feed}) that  $z_\rho\equiv0$. Again, invoking Lemma 3, one has $Q\equiv(\pm1,\bold{0}_3^T)^T$, or $Q\equiv(0,v_{\rho}^T)^T$. It is clear that the largest invariant set in $\Upsilon$, for the closed loop system, characterized by $\dot{V}=0$ is given by $\Omega_{inv}=\bigcup_{i=1}^4 \Omega_i$. \\
Since we showed that $\dot{V}\leq 0$, one has $V(\chi(t))\leq V(\chi(0))$, for all $t\geq 0$, which shows that $\Phi_1$ is a positively invariant sublevel set. Since $V(\chi)\geq 2\lambda_{min}(W_\gamma)\|\tilde{q}\|^2$ and  $V(\chi)\geq 2\lambda_{min}(W_\rho)\|q\|^2$, it is clear that $\Min_{\|\tilde{q}\|=1,~\|q\|=1}V(\chi)=2\min\{\lambda_{min}(W_\gamma),\lambda_{min}(W_\rho)\}$. Consequently, $\Phi_1 \subset \{\chi\in\Upsilon~|~\|\tilde{q}\|<1~,~\|q\|<1\}$. Therefore, $\Omega_2$, $\Omega_3$ and $\Omega_4$ do not belong to $\Phi_1$.
Finally, since the largest invariant set in $\Phi_1$, corresponding to $\dot{V}=0$, is nothing else but $\Omega_1$, the second claim of the theorem is proved.\\
Now, we need to show that  $\Omega_2$, $\Omega_3$ and $\Omega_4$ are unstable. First, let us focus our attention on $\Omega_2$ and $\Omega_3$ that belong to the manifold characterized by $(\tilde{q}_0=0,\omega=0)$. In fact, the closed loop dynamics of $\tilde{q}_0$ are given by $\dot{\tilde{q}}_0=(\tilde{q}^T W_{\gamma}\tilde{q})\tilde{q}_0-\frac{1}{2}\tilde{q}^T\hat{R}\omega$, which can be written, around $(\tilde{q}_0=0,\omega=0)$, as follows:
\begin{equation}
\dot{\tilde{q}}_0=\eta_\gamma \tilde{q}_0-\frac{1}{2}v_{\gamma}^T\hat{R}\omega,
\label{tilde_q0_dyn}
\end{equation}
where $\eta_\gamma:=v_{\gamma}^T W_{\gamma}v_{\gamma}$. Consider the following positive definite function: $V_0(\tilde{q}_0)=\frac{1}{2}\tilde{q}_0^2$, whose time derivative along the trajectories of (\ref{tilde_q0_dyn}) is given by
\begin{equation}
\begin{array}{rcl}
\dot{V}_0&=&\eta_\gamma\tilde{q}_0^2-\frac{1}{2}\tilde{q}_0v_{\gamma}^T\hat{R}\omega \geq \eta_\gamma\tilde{q}_0^2-\frac{1}{2}\|\omega\| |\tilde{q}_0|.
\end{array}
\end{equation}
Since $\tilde{q}_0$ and $\omega$ are bounded (as per (\ref{Lyap}) and (\ref{dLyap_control})), and $\tilde{q}_0=0$ is not an invariant manifold, it is clear that, in the neighborhood of $(\tilde{q}_0=0,\omega=0)$, there exists a finite parameter $k_\gamma>0$ such that $\frac{1}{2}\|\omega\| \leq k_\gamma |\tilde{q}_0|$. Consequently, one has $\dot{V}_0 \geq (\eta_\gamma-k_\gamma) \tilde{q}_0^2$.
Since $\eta_{\gamma}$ is an eigenvalue of $W_{\gamma}$, $\eta_{\gamma}\geq\lambda_{min}(W_{\gamma})>0$. Hence, $\dot{V}_0>0$ as long as $\lambda_{min}(W_{\gamma})>k_\gamma$ and $\tilde{q}_0\not=0$, which shows that $\Omega_2$ and $\Omega_3$ are unstable.\\
Now, let us show that $\Omega_4$ is unstable, using Chetaev arguments\footnote{The reader is referred to \cite{Khalil} for more details about Chetaev Theorem.}. To this end, we will show the instability of the manifold characterized by $(q_0=0,\omega=0)$. Let us define $\delta:= {q}^TJ\omega$, and consider the dynamics of ${q}_0$ and $\delta$ in the neighborhood of $\Omega_4$
\begin{equation}\label{linearization2}
\begin{array}{rcl}
\dot{q}_0&=&-\frac{1}{2}v_{\rho}^T \omega\\
\dot{\delta}&=&-2\eta_{\rho} q_0,\\
\end{array}
\end{equation}
where $\eta_{\rho}=v_{\rho}^TW_{\rho}v_{\rho}$. Let us consider the following Chetaev function: $\mathcal{V}(q_0,\delta)=-{q}_0\delta$.
Define the set $B_r=\{x:=(q_0,\delta)\in [-1~,~1]\times\mathbb{R}~|~\|x(t)\| < r \}$
where $0<r<1$. Let us also define a subset of $B_r$ where $\mathcal{V}>0$,  that is, $U_r=\{x\in B_r~| ~\mathcal{V}(x)>0\}$.
Note that $U_r$ is non-empty for all $0<r<1$, and $\mathcal{V}(0,0)=0$.  The time derivative of $\mathcal{V}$, in view of (\ref{linearization2}), is given by
\begin{equation}\label{der_chet}
 \begin{array}{rcl}
  \dot{\mathcal{V}}&=& 2\eta_{\rho} q_0^2+\frac{1}{2}\delta v_{\rho}^T\omega \geq (2\eta_{\rho}-\frac{1}{2}k_fk_q^2)q_0^2,
  \end{array}
  \end{equation}
where we used the fact that $|\delta|\leq k_f \|\omega\|$, with $\|J\|=k_f$, and the fact that, around the manifold characterized by $(q_0=0,\omega=0)$, with $q_0\not=0$, there exists a positive parameter $k_q$ such that $\|\omega\|\leq k_q |q_0|$. Since $\eta_{\rho}$ is an eigenvalue of $W_{\rho}$, $\eta_{\rho}\geq\lambda_{min}(W_{\rho})>0$. Therefore, picking  $\lambda_{min}(W_{\rho})>\frac{1}{4}k_fk_q^2=: k_{\rho}$, it follows that $\dot{V}>0$ for all $q_0\not=0$.\\
Picking the initial conditions $x(0)\in U_r$, $x(t)$ must leave $U_r$ since $\mathcal{V}(x)$ is bounded on $U_r$ and $\dot{\mathcal{V}}(x)>0$ everywhere in $U_r$. Since $\mathcal{V}(x(t))\geq \mathcal{V}(x(0))$,  $x(t)$ must leave $U_r$ through the circle $\|x\|=r$ and not through the edges $\mathcal{V}(x)=0$ (\textit{i.e.,} $\delta=0$ or $\tilde{q}_0=0$). Since this can happen for arbitrarily small $r$, the equilibrium $(q_0=0,\delta=0)$ is unstable and so is $\Omega_4$.\\
Finally, since the possible stable manifolds associated with the unstable equilibria belong to a set of Lebesgue measure zero in the state space, it is clear that $\Omega_1$ is almost globally attractive.
\end{proof}

%\begin{remark}
%Note that in this work the equilibria are represented in terms of  vector parts of the unit quaternion. However, due to the unit quaternion constraint, the reader should keep in mind that whenever the vector part is zero, the scalar part is $\pm1$.
%\end{remark}

\begin{remark}
For the sake of simplicity, the desired attitude has been taken as $I_3$. However, it is straightforward to extend the proposed control law to the case of an arbitrary constant desired attitude $R_d=\mathcal{R}(Q_d)$ by taking $z_\rho=\sum_{i=1}^n \rho_i S(R_d^Tr_i)b_i$.
%In fact, in this case, the expression of $z_\rho$ in Lemma 1 becomes $z_\rho=-2R_d^T(q^e_0I-S(q^e))W_{\rho} q^e$, where $q_0^e$ and $q^e$ are the components of the unit-quaternion error $Q^e=Q\odot Q_d^{-1}$ corresponding $RR_d^T$. Consequently, as per Lemma 3, $z_\rho=0$ in this case is equivalent to $(q^e_0=0,q^e=v_{\rho})$ or $(q^e_0=\pm1, q^e=0)$.
\end{remark}

\subsection{Attitude stabilization using preconditioned vector measurements}
In this section, we will show almost global asymptotic stability of the desired equilibrium point without any conditions on the minimum eigenvalues of $W_\gamma$ and $W_\rho$. This results are made possible through an appropriate preconditioning of the inertial measurement vectors. \\
Let us assume that we have only two vector measurements $b_1,b_2\in\mathbb{R}^3$ corresponding, respectively, to two inertial vectors $r_1,r_2\in\mathbb{R}^3$ (which are assumed non-collinear).  Let us define the normalized (unit) vectors $v_1=\frac{r_1}{\|r_1\|}$, $v_2=\frac{r_1 \times r_2}{\|r_1 \times r_2\|}$ and  $v_3=\frac{(r_1 \times r_2)\times r_1}{\|(r_1 \times r_2)\times r_1\|}$. It is clear that $\{v_1,v_2,v_3\}$ forms a unit orthonormal basis.
Let us define $u_1=\frac{b_1}{\|r_1\|}$, $u_2=\frac{b_1 \times b_2}{\|r_1 \times r_2\|}$ and  $u_3=\frac{(b_1 \times b_2)\times b_1}{\|(r_1 \times r_2)\times r_1\|}$. Using the fact that $b_1=R^Tr_1$ and $b_2=R^Tr_2$, one can easily show that $u_i=R^Tv_i$, for $i\in\{1,2,3\}$. Let us define  $z_\gamma:=\gamma\sum_{i=1}^3 S(\hat{u}_i)u_i$, $z_\rho:=\rho\sum_{i=1}^3  S(v_i)u_i$, $\hat{u}_i=\mathcal{R}(\hat{Q})^Tv_i$, $W_{\gamma}:=-\gamma \sum_{i=1}^{3}  S(v_i)^2$ and $W_{\rho}:=-\rho \sum_{i=1}^{3}  S(v_i)^2$, with $\gamma >0$, $\rho>0$.

The choice of $v_1$, $v_2$ and $v_3$ leads to $W_{\gamma}=2\gamma I_3$ and  $W_{\rho}=2\rho I_3$. This can be shown as follows: Since $\{v_1,v_2,v_3\}$ is a unit orthonormal basis, one has $M_\gamma:=\gamma \sum_{i=1}^3 v_iv_i^T=\gamma I_3$ since $M_\gamma v_i=\gamma v_i$, $i=1,2,3$. Furthermore, since $W_\gamma=\mu I_3-M_\gamma$ with $\mu=\gamma \sum_{i=1}^3  v_i^Tv_i=3\gamma$, it is clear that  $W_{\gamma}=2\gamma I_3$. \\
In this case, using Lemma 1 and the new values of $W_\gamma$ and $W_\rho$, it follows that
\begin{equation}\label{z_gr}
z_\gamma=-4\gamma\hat{R}^T\tilde{q}_0\tilde{q},~~~ z_\rho=-4\rho q_0q.
\end{equation}

Now, one can state our second theorem:
\begin{theorem}
Consider system (\ref{quat_dyn})-(\ref{rot_dyn}) under the control law (\ref{feed}) and the input of the auxiliary system (\ref{obser1}) given in (\ref{beta}), with $z_\gamma$ and $z_\rho$ as defined in subsection IV-C. Then,
\begin{itemize}
\item[$1)$] The equilibria of the closed-loop system are given by
     $\Psi_1=(0,0,0)$, $\Psi_2=(\tilde{q}\in \mathbb{S}^2,0,0)$, $\Psi_3=(\tilde{q}\in \mathbb{S}^2,q\in \mathbb{S}^2,0)$ and $\Psi_4=(0,q\in \mathbb{S}^2,0)$.
 \item[$2)$] The equilibria $\Psi_2$, $\Psi_3$ and $\Psi_4$ are unstable, and $\Psi_1$ is almost globally asymptotically stable.
\end{itemize}
\end{theorem}
\begin{proof}
As in the proof of Theorem 1, one can show that the closed-loop system is autonomous. Consider the positive definite function (\ref{Lyap}) which, due to the fact that $W_{\gamma}=2\gamma I_3$ and $W_{\rho}=2\rho I_3$,  can be written as
\begin{equation}\label{lyap_Th2}
V(\tilde{q},q,\omega)=4\gamma\tilde{q}^T\tilde{q}+4\rho q^Tq+\frac{1}{2}\omega^T J \omega=\frac{1}{2}\left(\gamma\Sum_{i=1}^{3} \tilde{u}_i^T  \tilde{u}_i+\rho \Sum_{i=1}^{3} e_i^T e_i+\omega^T J \omega \right),
\end{equation}
where $\tilde{u}_i:=\hat{u}_i-u_i$ and $e_i:=v_i-u_i$. Using the fact that $\dot{\tilde{u}}_i= S(\hat{u}_i)(\beta-\omega)+S(\tilde{u}_i)\omega$ and $\dot{e}_i=-S(u_i)\omega$, the time derivative of (\ref{lyap_Th2}), along the closed-loop system trajectories, is given by
\begin{equation}
  \dot{V}= -z_\gamma^T z_\gamma=-16\gamma^2\tilde{q}^T(I_3-\tilde{q}\tilde{q}^T)\tilde{q}=-16\gamma^2\|\tilde{q}\|^2(1-\|\tilde{q}\|^2),
  \end{equation}
which vanishes at $\tilde{q}=0$ or $\|\tilde{q}\|=1$. Similar steps as in the proof of Theorem 1, in view of (\ref{z_gr}), can be applied to conclude that the largest invariant set in $\Upsilon$, characterized by $\dot{V}=0$, is given by $\Psi_{inv}=\bigcup_{i=1}^4 \Psi_i$.
Again, as in the proof of Theorem 1, one can show that the equilibrium $\Psi_1$ is asymptotically stable with the domain of attraction containing the domain $\Phi_2=\left\{\chi\in \Upsilon~|~\chi^TP\chi < 4\min\{\gamma,\rho\} \right\}\subset \{\chi\in\Upsilon~|~\|\tilde{q}\|<1~,~\|q\|<1\}$.\\
Now, we need to show that $\Psi_2$, $\Psi_3$ and $\Psi_4$ are unstable. In fact, at the equilibria characterized by $\Psi_2$, the Lyapunov function becomes  $V_{\Psi 2}=4\gamma v^Tv= 4\gamma $,
where $v$ is an arbitrary vector living in the unit 2-sphere $\mathbb{S}^2$.
A small disturbance $\epsilon \in D$ acting on $v$, in the neighborhood of $\Psi_2$, \textit{i.e.,} $\tilde{q}= v + \epsilon$, such that $0\leq\|\tilde{q}\|<1$,  would lead to $V=4\gamma \|\tilde{q}\|^2 <4\gamma=V_{\Psi 2}$.
Since $V$ is non-increasing, it is clear that $\Psi_2$ is unstable. Similar arguments can be used to show that $\Psi_3$ and $\Psi_4$ are unstable as well, and hence omitted here.
Finally, almost global attractivity of $\Psi_1$ follows from the previous results and the fact that $\mathbb{S}^2$ has Lebesgue measure zero on $D$.
%and $(\tilde{q}_0,q_0)=(0,0)$ has Lebesgue measure zero on $[-1,1]\times[-1,1]$.
\end{proof}

\begin{remark}
The proposed approach guarantees \textit{a priori} boundedness of the control input, that is, $\|\tau\|\leq \sum_{i=1}^n (\gamma_i+\rho_i) \|r_i\|^2$ in Theorem 1, and $\|\tau\|\leq 3(\gamma+\rho)$ in Theorem 2.
\end{remark}

\begin{remark}
The control law (\ref{feed}) can be written as $\tau (\tilde{Q},Q)=-2\hat{R}^T(\tilde{q}_0I-S(\tilde{q}))W_{\gamma}\tilde{q}-2(q_0I-S(q))W_{\rho} q$,
with $\hat{R}:=\mathcal{R}(\tilde{Q})^T\mathcal{R}(Q)$. Note that $\tau (\tilde{Q},Q)=\tau (-\tilde{Q},-Q)$, which confirms the fact that the proposed control strategy is not affected by the sign ambiguity of the unit-quaternion representation since our approach does not rely on any attitude reconstruction.
\end{remark}

\begin{remark}\label{remIMU}
In practical applications involving small scale aerial vehicles, for instance, it is customary to equip the vehicle with an IMU composed of accelerometers, magnetometers and gyroscopes. The gyroscopes provide the angular velocity $\omega$, the magnetometers provide a vector measurement of the earth magnetic field in the body attached frame $m_{\mathcal{B}}$, which is related to the earth's magnetic field $m_{\mathcal{I}}$ expressed in the inertial frame through $\bar{m}_{\mathcal{B}}=Q^{-1}\odot \bar{m}_{\mathcal{I}} \odot Q$. The accelerometers provide a vector measurement of the apparent acceleration $ a_{\mathcal{B}}$ in the body attached frame, which is related to the acceleration $ a_{\mathcal{I}}$ expressed in the inertial frame through $\bar{a}_{\mathcal{B}}=Q^{-1}\odot \bar{a}_{\mathcal{I}} \odot Q$. In the case of quasi-stationary flights (\textit{i.e.,} $||\dot{v}|| \ll g$), the acceleration expressed in the inertial frame is given by $a_{\mathcal{I}}=[0,0,-g]^T$. The attitude stabilization controllers in Theorem 1 and Theorem 2, could be implemented using the vector measurements obtained from the IMU ($a_{\mathcal{B}}, m_{\mathcal{B}}$), taking $r_1=a_{\mathcal{I}}$, $b_1=a_{\mathcal{B}}$, $r_2=m_{\mathcal{I}}$ and $b_2=m_{\mathcal{B}}$.
\end{remark}

\section{Simulation Results}
In this section, we present some simulation results showing the effectiveness of the proposed approach. The inertia matrix has been taken as $J=\mbox{diag}(0.5,0.5,1)$ and the inertial vectors as $r_1=[0,0,1]^T$ and $r_2=[1,0,1]^T$. The control parameters have been taken as $\gamma_1=\gamma_2=10$ and $\rho_1=\rho_2=0.5$. We performed two simulation tests, using the control law in Theorem 1, to show the performance of the proposed control scheme and confirm the avoidance of the unwinding phenomenon.
In the fist simulation test, we considered the following initial conditions: $\omega(0)=[0,0,0]^T$, $Q(0)=[0.8,0,0,0.6]^T$ and $\hat{Q}(0)=[1,0,0,0]^T$. In the second simulation test, we considered the same initial conditions except for $Q$, where we started the scalar part of the unit quaternion from a negative value, \textit{i.e.,}  $Q(0)=[-0.8,0,0,0.6]^T$.
Figure 1 and Figure 3 show the evolution of the three components of the angular velocity with respect to time for Test 1 and Test 2, respectively. Figure 2 and Figure 4 show the evolution of the unit-quaternion with respect to time for Test 1 and Test 2, respectively. We can clearly see that the unwinding phenomenon is avoided since both equilibria $(q_0=1,q=0,\omega=0)$ and $(q_0=-1,q=0,\omega=0)$ are asymptotically stable.
\begin{figure}[t]
\begin{minipage}[b]{0.5\columnwidth}
\centering
\includegraphics[height=5cm]{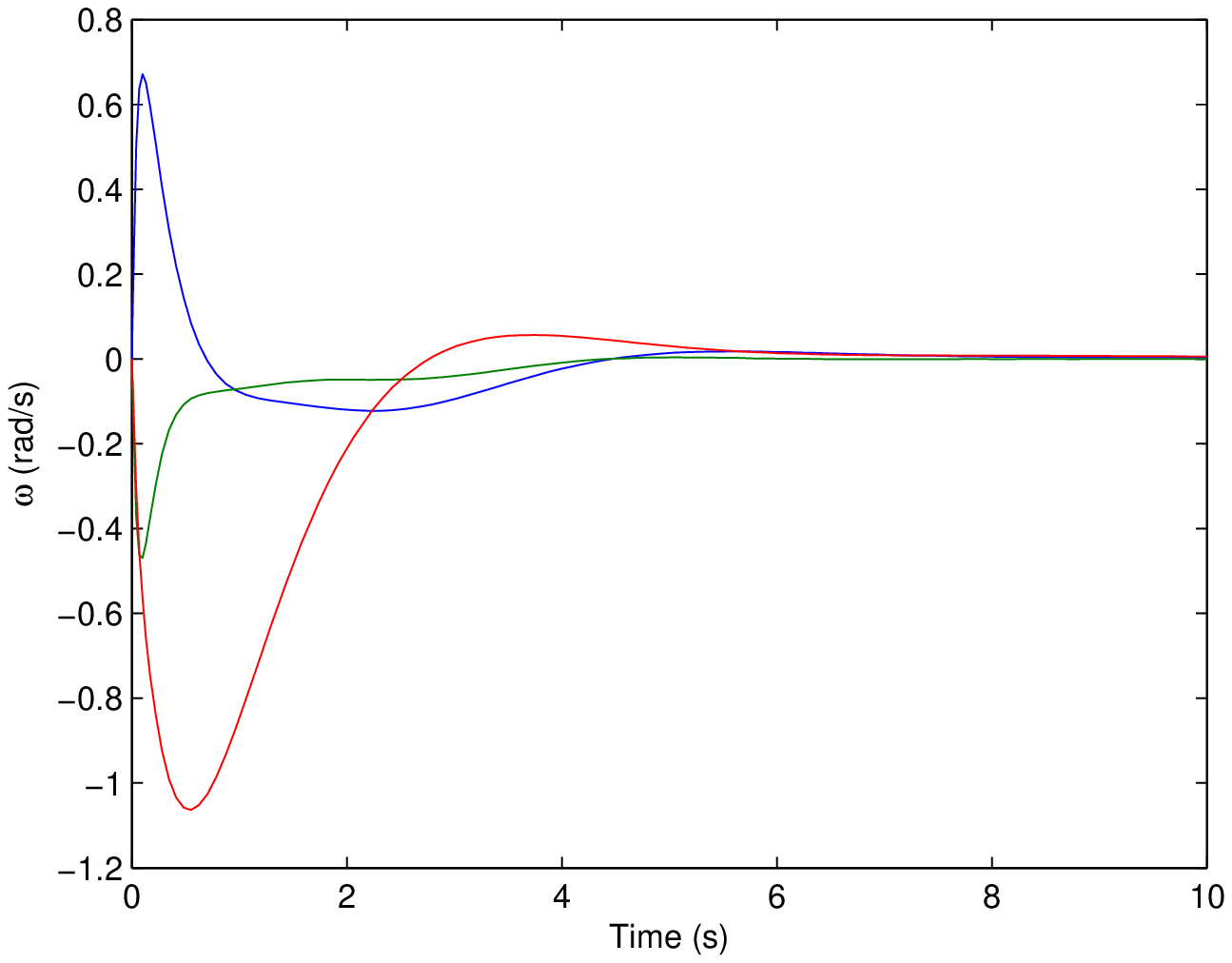}
\caption{\footnotesize  Test 1: Angular velocity vs. time}
\end{minipage}
\hfill
\begin{minipage}[b]{0.5\columnwidth}
\centering
\includegraphics[height=5cm]{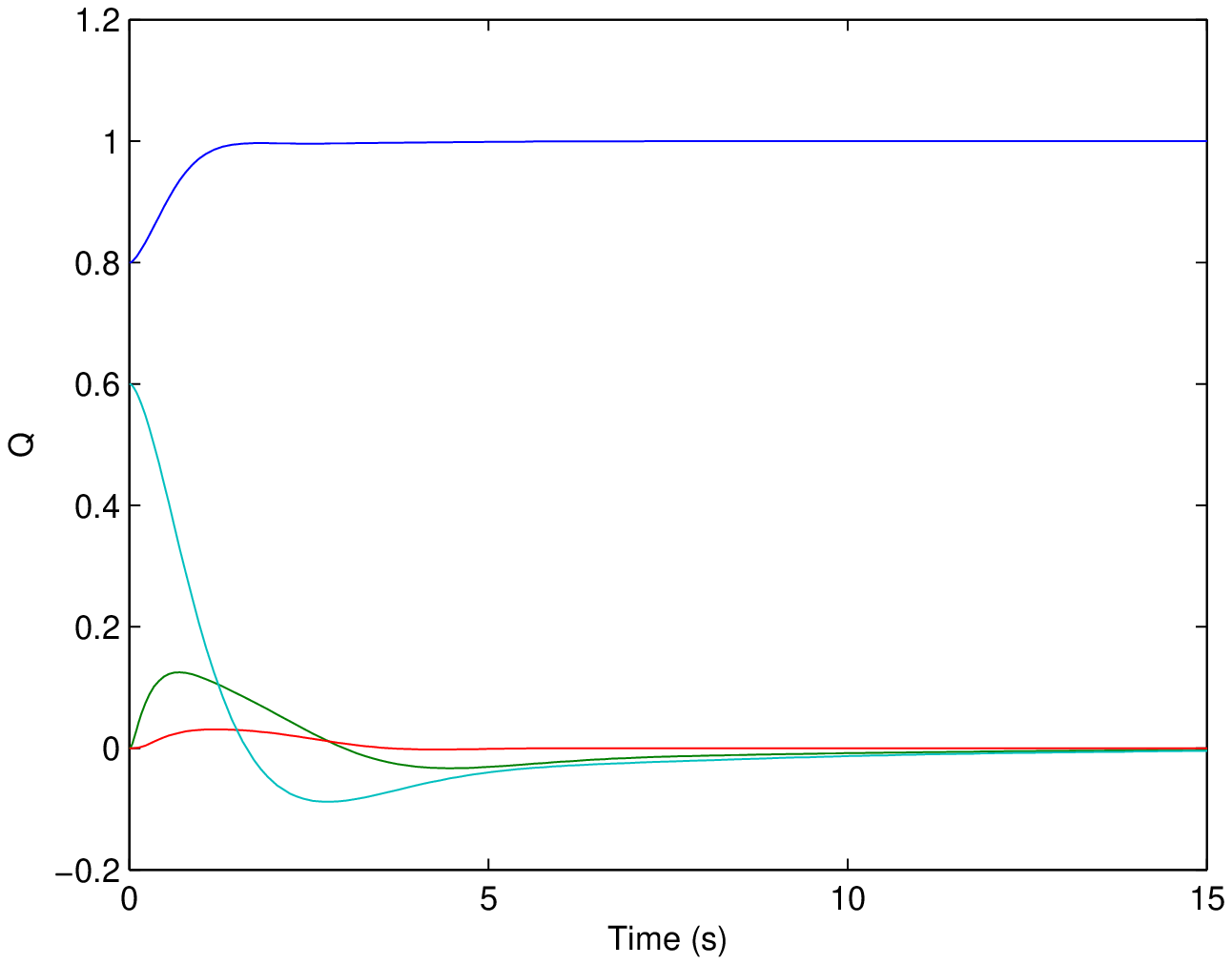}
\caption{\footnotesize Test 1: Unit-quaternion vs. time}
\end{minipage}
\end{figure}

\begin{figure}[t]
\begin{minipage}[b]{0.5\columnwidth}
\centering
\includegraphics[height=5cm]{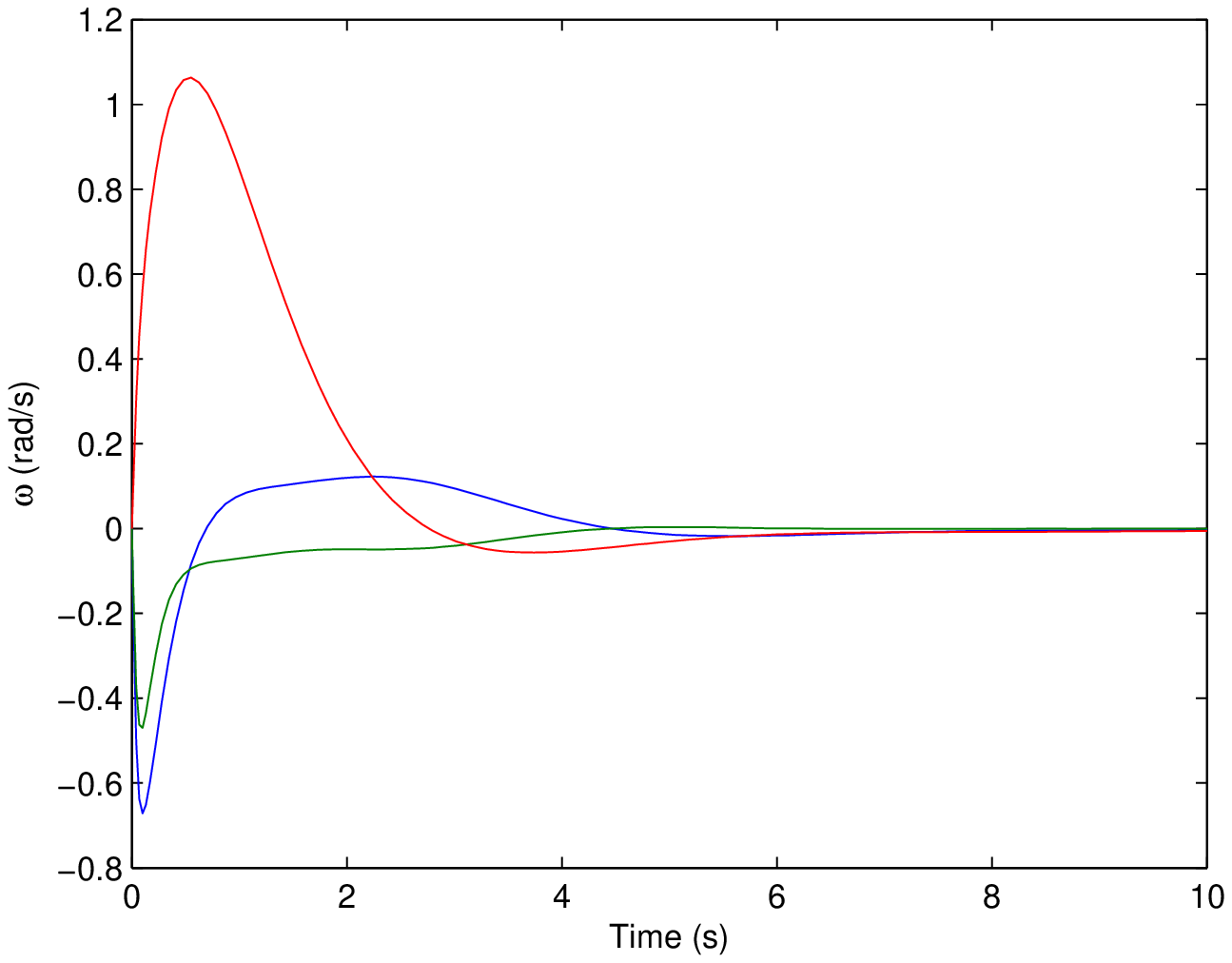}
\caption{\footnotesize  Test 2: Angular velocity vs. time}
\end{minipage}
\hfill
\begin{minipage}[b]{0.5\columnwidth}
\centering
\includegraphics[height=5cm]{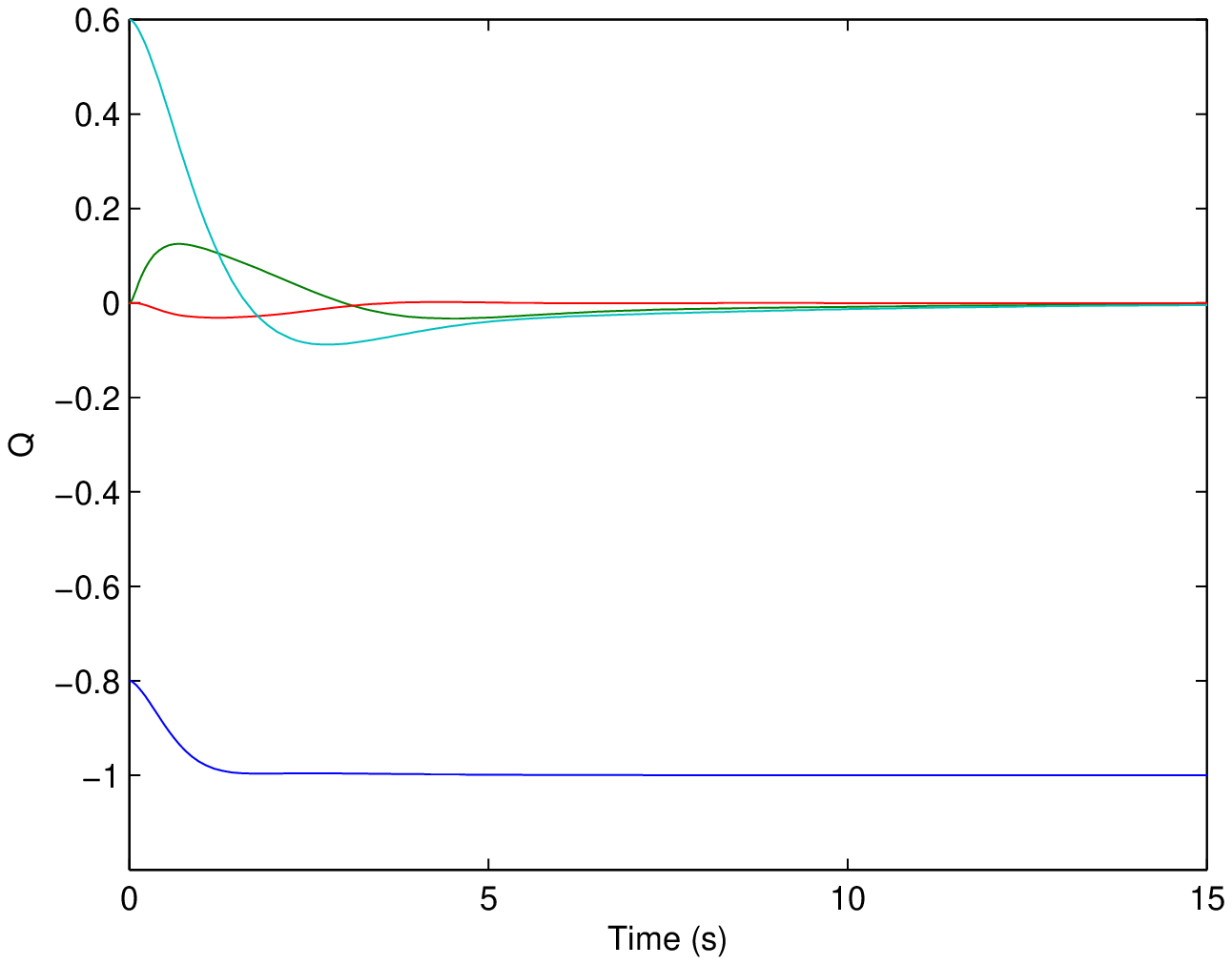}
\caption{\footnotesize Test 2: Unit-quaternion vs. time}
\end{minipage}
\end{figure}

\section{Concluding discussion}
 It is important to point out that the present work was motivated by the fact that the existing attitude control schemes (without angular velocity measurement) use explicitly the orientation of the rigid body which is not directly measurable by any means. Therefore, an additional algorithm is required for attitude reconstruction (or estimation) using the available measurements (\textit{e.g.,} IMU measurements). Reconstructing the attitude from inertial vector measurements via static optimization algorithms is often hampered by two problems: 1) Sensitivity to noise; 2) Non-robustness due to a memoryless (random) selection among the two possible quaternion solutions $Q$ and $-Q$ (which represent exactly the same orientation) as pointed out in \cite{Mayhew}. On the other hand, dynamic attitude estimation algorithms of Kalman-type or complementary-filtering-type require the use of the angular velocity measurement. Hence, the integration of such algorithms with a velocity-free attitude controller is questionable since the angular velocity is assumed to be unavailable. Besides this fundamental circular-reasoning-like problem, the stability of the overall system (observer-controller) is not generally guaranteed as the separation principle does not generally hold for nonlinear systems.\\
 Motivated by this facts, we proposed a velocity-free attitude stabilization scheme that does not require the orientation reconstruction. To the best of our knowledge, this work is the first incorporating directly vector measurements in the design of velocity-free attitude controllers without attitude reconstruction.\\
 The second point that is worth mentioning, is the fact that our control scheme guarantees almost global asymptotic stability of the equilibrium points characterized by $(q_0=\pm1,q=0,\omega=0)$ which clearly avoids the well-known unwinding phenomenon (as shown in the simulation results). Usually unwinding is generated when the equilibria at $q_0=1$ and $q_0=-1$ are not both stable, which causes some trajectories, initialized around the unstable equilibrium, to undergo an unnecessary motion before joining the stable equilibrium, while both equilibria represent the same physical orientation.

\end{document}